\documentclass[11pt]{article}
\usepackage[utf8]{inputenc}
\usepackage[english]{babel}
\usepackage{times}
\usepackage{paralist}
\usepackage{amsmath,amstext,amssymb,amsthm, amsfonts}
\usepackage{graphicx}
\usepackage{xcolor}
\newtheorem{lemma}{Lemma}[section]
\newtheorem{theorem}[lemma]{Theorem}
\newtheorem{corollary}[lemma]{Corollary}

\newtheorem{proposition}[lemma]{Proposition}

\newcommand{\epi} {\mathop{\rm epi}}
\newcommand{\outlim} {\mathop{\rm LimOut\,}}
\newcommand{\innlim} {\mathop{\rm LimInn\,}}
\newcommand{\slim} {\mathop{\rm lim\,sup\,}}
\newcommand{\ilim} {\mathop{\rm lim\,inf\,}}
\newcommand{\elim}{\mathop{\rm e\mbox{-}lim}}

\def\U{\mathbb{U}}

\def\N{\mathbb{N}}

\def\h{\mathbf{I}}
\def\X{\mathbb{X}}
\def\E{\mathbb{E}}

\def\R{\mathbb{R}}

\def\B{\mathcal{B}}
\def\ball{\mathbb{B}}

\def\PP{\mathcal{P}}

\def\RR{\overline{\mathbb{R}}}

\def\nnmin{\mathop{\rm minimize}}

\begin{document}

\begin{center}
\begin{large}
{\bf Epi-Convergence of Expectation Functions under Varying Measures and Integrands}
\smallskip
\end{large}
\vglue 0.3truecm
\begin{tabular}{c}
  {\sl Eugene A. Feinberg}\\
  Department of Applied Mathematics and Statistics, Stony Brook University\\
  {\sl Pavlo O. Kasyanov}\\
  Institute for Applied System Analysis, National Technical University of Ukraine\\
  ``Igor Sikorsky Kyiv Polytechnic Institute''\\
  {\sl Johannes O. Royset}\\
  Operations Research Department, Naval Postgraduate School
\end{tabular}

\vskip 0.4truecm

\today

\end{center}

\vskip 0.3truecm

\noindent {\bf Abstract}. For expectation functions on metric spaces, we provide sufficient conditions for epi-convergence under varying probability measures and integrands, and examine applications in the area of sieve estimators, mollifier smoothing, PDE-constrained optimization, and stochastic optimization with expectation constraints. As a stepping stone to epi-convergence of independent interest, we develop parametric Fatou's lemmas under mild integrability assumptions. In the setting of Suslin metric spaces, the assumptions are expressed in terms of Pasch-Hausdorff envelopes. For general metric spaces, the assumptions shift to semicontinuity of integrands also on the sample space, which then is assumed to be a metric space.

\vskip 0.2truecm


\baselineskip=15pt

\section{Introduction}

It is well known that epi-convergence is a suitable framework for analyzing approximating optimization problems. It furnishes guarantees about near-minimizers, level-sets, and minimum values for the problems as compared to those of a limiting, actual problem; see for example \cite{Attouch.84,BRW,VaAn}. A recent summary of properties for functions defined on metric spaces is given by \cite{RoysetWets.19}. Epi-convergence is therefore central in the development of optimization algorithms, which often rely on approximations, and in stability analysis of optimization problems. It also emerges as a means to establish strong duality for dual problems obtained via general perturbation schemes \cite[Section 5.F]{primer}. In this paper, we provide sufficient conditions for epi-convergence of expectation functions under varying measures and integrands.

The study of expectation functions is critical for stochastic optimization and statistical estimation. The first systematic examination of epi-convergence for such functions appears to be \cite{DupacovaWets}, which also includes a review of the earlier literature. The setting is limited to expectation functions defined on $\R^n$, a fixed integrand, and probabilities converging weakly under a tightness assumption. An extension to separable reflexive Banach spaces, with a focus on Mosco-epi-convergence, is developed in \cite{LucchettiWets.93} under similar assumptions. For separable Banach spaces, \cite{AttouchWets.90} establishes epi-convergence for empirical distributions (generated by pairwise independent sampling) and a fixed integrand, and also includes extensions to convergence in the stronger Attouch-Wets topology. A further extension to complete separable metric spaces is given by \cite{ArtsteinWets.95}, again under a fixed integrand and empirical distributions obtained by independent sampling. A nearly identical setting is considered in \cite{Hess}, which also develops statistical applications. A specialization to the case of convex integrands is provided by \cite{KingWets.91}; see also \cite{SurBirge}. Going beyond empirical distributions, \cite{KorfWets.01,Seri2} establish epi-convergence of expectation functions for ergodic processes. Recently, \cite{Diem} reexamines epi-convergence under weakly converging probabilities as in \cite{DupacovaWets}, but now in the general setting of Hausdorff topological spaces. Without making explicit assumptions on the sequence of probabilities, but instead expressing assumptions in terms of Pasch-Hausdorff envelopes, \cite{Seri3} develops a general framework for establishing epi-convergence of expectation functions defined on separable metric spaces.

While we make precise comparison below, the present paper is distinct by considering {\em varying integrands}; the above studies keep the integrand fixed as the probabilities change. Varying integrands appear in applications, for example, as the result of computationally motivated approximations that need to be consider in conjunction with the varying probabilities. Sieves in statistical estimation  \cite{GemanHwang.82, RoysetWets.20} as well as discretization of an infinite-dimensional space of controls and numerical solution of differential equations \cite{SchwartzPolak.96,PhelpsRoysetGong.16,ChenRoyset.22b} can be viewed from this angle. Moreover, we establish conditions under which a lower limit of approximating expectation functions is bounded from below by the actual expectation function in the sense of a ``parametric'' Fatou's lemma with varying measure. This property furnishes the ``first half'' of epi-convergence (a limit gives the ``second half''), but we argue that it stands on its own as a useful property, especially in difficult cases when epi-convergence might be beyond reach.

To formalize the setting, we let $\RR:=\R \cup \{-\infty, +\infty\}$ and denote by $\B(\RR)$ its Borel $\sigma$-algebra. For probability space $(\Omega,\mathcal{A},\mu)$ and $(\mathcal{A},\B(\RR))$-measurable function $g:\Omega\to\RR$, the integral
\[
\int_\Omega g(\omega) \mu(d\omega) := \int_\Omega g_+(\omega) \mu(d\omega) - \int_\Omega g_-(\omega) \mu(d\omega),
\]
where $g_+(\omega) := \max\{g(\omega), 0\},$ $g_-(\omega) := - \min\{g(\omega), 0\},$ and we use the conventions $+\infty - \alpha = +\infty$ and $\beta - (+\infty) = -\infty$ hold for $\alpha \in \RR$ and $\beta\in \R$ everywhere in the paper. This integral is an expectation, and we often use the notation
\[
\E^\mu\big[g(\omega)\big] := \int_\Omega g(\omega) \mu(d\omega).
\]

Let $\hat{\mathcal{A}}$ be the $\mu$-completion of the $\sigma$-algebra $\mathcal{A}.$  Then the integral is always defined for an $\hat{\mathcal{A}}$-measurable function $g:(\Omega,\hat{\mathcal{A}})\to (\RR,\mathcal{B}(\RR)),$ and, if both integrals of $g_+$ and $g_-$ are equal to $+\infty,$ then the integral of $g$ is also equal to $+\infty.$

Let $(\X,d_\X)$ be a metric space and $(\Xi,\mathcal{F})$ be a measurable space. A function $f:\Xi\times \X\to \RR$ is an {\em integrand} if $f(\cdot,x)$ is $(\mathcal{F},\B(\RR))$-measurable for each $x\in \X$. An integrand $f$ together with a probability $P$ on $(\Xi,\mathcal{F})$ define an {\em expectation function} $E^P[f]:\X\to \RR$ given by
\[
E^P\big[f\big](x) := \E^{P}\big[f(\xi,x)\big],\quad x\in \X.
\]
If $\Xi$ is a metric space, we always consider $\mathcal{F}$ being its Borel $\sigma$-algebra $\B(\Xi).$

Let $\N = \{1, 2, \dots\}$. In this paper, we consider sequences of integrands $\{f, f^\nu:\Xi\times \X\to \RR, \nu\in\N\}$ and probabilities $\{P, P^\nu, \nu\in\N\}$, all defined on $(\Xi, \mathcal{F})$, and examine sufficient conditions for epi-convergence of the resulting sequence of expectation functions $\{E^{P^\nu}[f^\nu]:\X\to \RR, \nu\in\N\}$ to the expectation function $E^P[f]:\X\to \RR$. We also examine when $x^\nu\to x$ implies $\ilim_{\nu\to +\infty} E^{P^\nu}[f^\nu](x^\nu) \geq E^P[f](x)$, which can be viewed as a ``functional'' Fatou's lemma under varying measure.

Section 2 presents preliminary facts. The main theorems appear in Section 3. Section 4 provides motivating applications. An appendix provides an additional proof.

\section{Preliminaries}\label{sec:prelim}

We recall that $\{h^{\nu}:\X\to \RR, \nu\in \N\}$ \textit{epi-converges} to $h:\X\to\RR,$ written as $h = \elim_{\nu \to +\infty} h^{\nu}$, if at each $x\in \X,$ the following two conditions hold:
\begin{itemize}
\item[{\rm(i)}] if $\{x^{\nu}\}_{\nu\in\mathbb{N}}$  converges to $x,$ then $h(x)\le \ilim\limits_{\nu \to +\infty}h^{\nu}(x^{\nu});$
\item[{\rm(ii)}] there exists a sequence $\{x^{\nu}\}_{\nu\in\mathbb{N}}$ convergent to $x$ such that $h(x)=\lim\limits_{\nu \to +\infty}h^{\nu}(x^{\nu}).$
\end{itemize}

For a function $h:\X\to \RR$, its {\em epigraph} is given by $\epi h := \{(x,\alpha)\in \X\times \R~|~h(x) \leq \alpha\}$. The function is {\em lower semicontinuous} (lsc) if $\epi h$ is a closed subset of $\X\times \R$ in the product topology. It is {\em upper semicontinuous} (usc) if $-h$ is lsc. The {\em lower regularization} of $h$, denoted by $\underline{h}:\X\to \RR$, is defined as
\begin{equation*}\label{eq:reg}
\underline{h}(x) := \ilim\limits_{x'\to x}h(x'),\quad x\in\X.
\end{equation*}
Thus, $\underline{h}:\X\to\RR$ is lsc. If $h$ is lsc, then $\underline{h}=h$.

The {\em Pasch-Hausdorff envelope} of $h:\X\to \RR$ with parameter $\kappa\in [0,+\infty)$ is the function $h_\kappa:\X\to \RR$ defined as
\begin{equation}\label{eq:appdefiop1}
h_\kappa(x) := \inf_{x'\in \X} \big\{ h(x') + \kappa d_{\X}(x,x')\big\},\quad x\in \X.
\end{equation}
Trivially, when $0\leq \kappa_1\leq \kappa_2<+\infty$, one has
\begin{equation*}\label{eq:app3}
h_{\kappa_1}(x) \le h_{\kappa_2}(x) \le h(x)~~~\forall x\in \X.
\end{equation*}
The definition of an infimum enables us to conclude that $h_\kappa(x) = \underline{h}_\kappa(x)$ for each $\kappa\in [0,+\infty)$ and each $x\in \X,$ where $\underline{h}_\kappa$ is the Pasch-Hausdorff envelope of $\underline{h}.$
If $\epi h\neq\emptyset$ and there are $\alpha>0$, $\beta\in\R$, and $\bar x\in \X$ such that $h(x) + \alpha d_\X(x,\bar x) + \beta\geq 0$ for all $x\in \X$, then (see for example \cite[Proposition~3.3]{Hess})
\begin{equation}\label{eq:app31}
\forall x\in \X: ~h_{\kappa}(x) \uparrow \underline{h}(x)\mbox{ as } \kappa\to +\infty
\end{equation}
and $h_\alpha(\bar{x})>-\infty$. In fact, for $\kappa> \alpha$, $h_\kappa$ is real-valued and Lipschitz continuous with modulus $\kappa$, i.e.,
\begin{equation}\label{eq:app4}
\big|h_{\kappa}(x) - h_{\kappa}(x')\big| \le \kappa d_{\X}(x,x') ~~~\forall x,x'\in \X.
\end{equation}

Furthermore, if $\{h^\nu:\X\to \RR, \nu\in \N\}$ is a sequence of functions such that for some $\kappa_0\in\mathbb{N}$ and $x_0\in\X,$
\begin{equation}\label{eq:app5}
\ilim_{\nu \to +\infty} h^\nu_{\kappa_0}(x_0) > -\infty,
\end{equation}
then \cite[Proposition~3.4]{Hess} implies that for each $x\in\X$
\begin{equation}\label{eq:app6}
\ilim\limits_{(\nu,y)\to(+\infty,x)} h^\nu(y) = \sup_{\kappa\in\mathbb{N}}\ilim_{\nu \to +\infty} h^\nu_{\kappa}(x).
\end{equation}

For $f:\Xi\times\X\to\RR$, $\xi\in \Xi$, and $\kappa\in [0,+\infty)$, we abuse notation slightly and write $f_\kappa$ for the Pasch-Hausdorff envelope of $f(\xi,\cdot)$, i.e., $f_\kappa:\Xi\times \X\to \RR$ is given by $f_\kappa(\xi,x) = \inf_{x'\in \X}\{ f(\xi,x') + \kappa d_\X(x',x)\}$. In view of  \cite[Propositions 4.2]{Hess},  if $\X$ is a Suslin space and the function $f$ is measurable, then for each $x\in\X$ the function 
$\xi\mapsto f_\kappa(\xi,x)$ is $\hat{\mathcal{F}}$-measurable, where $\hat{\mathcal{F}}$ is the $P$-completion of the sigma-algebra $\mathcal{F}.$ Likewise, $f_\kappa^\nu$ stems from $f^\nu:\Xi\times \X\to \RR$.

For probability spaces   $(\Xi, \mathcal{F},P^\nu),$ $\nu\in\N,$    let us consider the probability spaces  $(\Xi, \hat{\mathcal{F}}^\nu,P^\nu),$  where $\hat{\mathcal{F}}^\nu$ are  $P^\nu$-completions of the  $\sigma$-algebra $\mathcal{F}$.  As explained in the previous paragraph, if  $f^\nu:\Xi\times\X\to\RR$ are  measurable functions, $\nu\in\N,$ then for each $\kappa, \nu\in\N$ and for each $x\in\X,$ the function
 $\xi\mapsto f^\nu_\kappa(\xi,x)$ is $\hat{\mathcal{F}}^\nu  $-measurable on $\Xi$ if $\X$ is a Suslin space.

\section{Main Results}

This section provides two sufficient conditions for epi-convergence of expectation functions. Theorem~\ref{th:general_epi} extends \cite[Theorem 3.1]{Seri3} by considering varying integrands. As in \cite{Seri3}, the main assumptions are expressed in terms of Pasch-Hausdorff envelopes. Theorem~\ref{th:epi} establishes epi-convergence in the setting with weakly converging probabilities and extends \cite[Theorem 1]{Diem} to varying integrands. (While we restrict the attention to  metric spaces, \cite{Diem} deals with Hausdorff topological spaces; see below for a detailed comparison.) Here, we dispense of Pasch-Hausdorff envelopes and instead rely on Fatou's lemmas for weakly convergent probabilities \cite{UFL,FKL2,FKL18}.

In contrast to \cite{Seri3}, \cite{Diem}, and earlier studies, we also bring forth bounds on lower epi-limits, which emerge as useful in their own right as demonstrated in Section 4. These bounds appear as Lemma~\ref{lm:general_lsc} and Theorem~\ref{th:weak-axill}.\\

\begin{lemma}\label{lm:general_lsc}{\rm(parametric Fatou's lemma with varying probabilities)} Let $(\Xi,\mathcal{F})$ be a measurable space, and $\X$ be a Suslin space. For probability measures $\{P, P^\nu, \nu\in \N\}$ on $(\Xi, \mathcal{F})$ and functions $\{f,f^\nu: \Xi\times \X\to \RR, \nu\in \N\}$, suppose that:
\begin{itemize}
\item[{\rm(i)}] $\{f, f^\nu, \nu\in\N\}$ are measurable functions, and $f(\xi,\,\cdot\,)$ is lsc for $P$-a.e. $\xi\in\Xi$;

\item[{\rm{(ii)}}] there exist a countable dense subset $\X_0$ of $\X$ and  an integer $\kappa_0\in\N$ such that
  \begin{equation}\label{eq:ass1}
  \liminf_{\nu \to +\infty} \, E^{P^\nu} \big[
      f_\kappa^\nu\big](x_0) \geq E^{P}\big[f_\kappa\big](x_0)>-\infty
    \end{equation}
for each $x_0\in\X_0$ and each $\kappa\in[\kappa_0,+\infty)$.
\end{itemize}
Then for every $x\in\X$
\[
\ilim\limits_{(\nu,y)\to(+\infty,x)} \, E^{P^\nu}\big[f^\nu\big](y) \ge E^P \big[f\big](x),
\]
with the right-hand side also exceeding $-\infty$ when there exists $\tilde x\in \X$ such that
\[
\ilim\limits_{(\nu,y)\to(+\infty,\tilde x)} \, E^{P^\nu}\big[f^\nu\big](y) < +\infty.
\]
 \end{lemma}

\begin{proof}
The proof of the lemma rethinks the proof of \cite[Theorem~3.1]{Seri3}; see Appendix~\ref{app}.
\end{proof}

We next turn to epi-convergence.\\

\begin{theorem}\label{th:general_epi}{\rm(epi-convergence of expectation functions)}  Let $(\Xi,\mathcal{F})$ be a measurable space, and $\X$ be a Suslin space.  For probability measures $\{P, P^\nu, \nu\in \N\}$ on $(\Xi, \mathcal{F})$ and functions $\{f,f^\nu: \Xi\times \X\to \RR\}$, suppose that assumptions (i)-(ii) of Lemma~\ref{lm:general_lsc} hold. Moreover, suppose that there is a countable dense subset $\X^0$ of $\X$ such that for each $x^0\in \X^0$,
  \[
  \limsup_{\nu \to +\infty} E^{P^\nu} \big[f^\nu\big](x^0) \leq E^{P}\big[f\big](x^0).
  \]
  Then
  \[
  \elim\limits_{\nu \to +\infty} E^{P^\nu}\big[{f^\nu}\big] = E^P\big[f\big].
  \]
 \end{theorem}
\begin{proof}
The liminf-condition in the definition of epi-convergence is immediately furnished by Lemma~\ref{lm:general_lsc}. The limit-condition of the definition is established in the second part of the proof of \cite[Theorem~3.1]{Seri3} with minor technical clarifications.
\end{proof}

Next, we turn to the special case when the probabilities $P^\nu$ converge weakly to $P$, which is common in applications. As is well known, weak convergence may arise for empirical distributions, in the context of central limit theorems, and for ergodic processes. But, weak convergence may also define a topology on spaces of distribution functions for the analysis of distributionally-robust optimization problems \cite{RoysetWets.17b}.

We denote by  $\h B$ the indicator of the event $B,$ that is $\h B=1$ if $B$ is True, and $\h B=0$ if $B$ is False.

\begin{theorem}\label{th:weak-axill}{\rm(parametric Fatou's lemma for weakly convergent probabilities)} Let $\X$ and $\Xi$ be metric spaces.
For probability measures $\{P, P^\nu, \nu\in \N\}$ on $(\Xi, \B(\Xi))$ and functions $\{f,f^\nu: \Xi\times \X\to \RR\}$, suppose that $P^\nu$ converges weakly to $P$ and the following hold for each $x\in \X$:
\begin{itemize}
\item[{\rm(i)}] $\{f(\,\cdot\,,x), f^\nu(\,\cdot\,,x), \nu\in\mathbb{N}\}$ are measurable;

\item[{\rm(ii)}] one has
\begin{equation}\label{eq:ui}
\ilim_{K\to+\infty}\ilim_{(\nu,y)\to(+\infty,x)} \E^{P^\nu}\big[ f^\nu(\xi,y)\h\big\{\xi\,:\, f^\nu(\xi,y)\le-K\big\}\big]=0;
\end{equation}

\item[{\rm(iii)}] for $P\mbox{-a.e. }\xi\in\Xi$ 
\begin{equation}\label{eq:conv1}
\ilim_{(\nu,y,\zeta)\to(+\infty,x,\xi)} f^\nu(\zeta,y) \ge f(\xi,x).
\end{equation}
\end{itemize}
Then for  each $x\in\X$ 
\begin{equation}\label{eq:conv2}
\ilim\limits_{(\nu,y)\to(+\infty,x)}E^{P^\nu}\big[f^\nu\big](y)\ge E^P \big[f\big](x).
\end{equation}
\end{theorem}

We observe that the first lower limit in \eqref{eq:ui} can be replaced by a limit because the subsequent expression is nondecreasing in $K$. The same observation applies to formula \eqref{eq:uief} below.

Before proving Theorem~\ref{th:weak-axill}, we formulate and prove Theorem~\ref{th:LFatou} which is a particular case of Theorem~\ref{th:weak-axill}  when $\X$ is a singleton. Theorem~\ref{th:LFatou} is an extension of Fatou's lemma for weakly converging probabilities and unbounded functions \cite[Theorem~2.4]{FKL18} to the convention $\E^P [g(\xi)]=+\infty$ if $\E^P[ g_+(\xi)]=\E^P[ g_-(\xi)]=+\infty.$ In \cite[Theorem~2.4]{FKL18}, the generalized Fatou's inequality \eqref{eq:new1} is proved under the assumption that each expectation in the inequality has its positive part and/or its negative part being finite. Another formal difference between Theorem~\ref{th:LFatou}  and \cite[Theorem~2.4]{FKL18} is that the former deals with probability measures while the latter deals with finite measures.  However, the case of finite measures follows from the case of probability measures by normalization, if the limiting finite measure of the entire space $\Xi$ is positive. If the limiting finite  measure is $0$, then the right-hand side of inequality \eqref{eq:new1} is $0$, and assumption \eqref{eq:uief} implies that the lower limit on the left-hand side of \eqref{eq:new1} is nonnegative. Thus, the extension of Theorem~\ref{th:LFatou} from probabilities to finite measures is routine.

\begin{theorem}\label{th:LFatou}{\rm(extended Fatou's lemma for weakly convergent probabilities; cp. \cite[Theorem~2.4]{FKL18} )}
Let $\Xi$ be a metric space, $\{P^\nu\}_{\nu\in\N}$ be a sequence of probabilities on $\Xi$ converging weakly to a probability $P$ on  $\Xi,$ and
	$\{ h^\nu\}_{\nu\in\N}$ be a sequence of measurable
	$\RR$-valued functions on $\Xi$ such that
\begin{equation}\label{eq:uief}
\ilim_{K\to+\infty}\ilim_{\nu \to +\infty} \E^{P^\nu}\big[  h^\nu(\xi)\h\big\{\xi\,:\, h^\nu(\xi)\le-K\big\}\big]=0.
\end{equation}
Then
\begin{equation}\label{eq:new1}
\ilim\limits_{\nu \to +\infty} \E^{P^\nu}\big[h^\nu(\xi)\big]\ge  \E^P\big[ \ilim_{(\nu,\zeta)\to(+\infty,\xi)} h^\nu(\zeta) \big].
\end{equation}
\end{theorem}
The proof of Theorem~\ref{th:LFatou} uses the following lemma stating additional properties of expectations of functions satisfying assumption \eqref{eq:uief}, which amounts to asymptotic uniform integrability from below.

\begin{lemma}\label{lm:new}
Under the assumptions of Theorem~\ref{th:LFatou}, either
\[
\ilim\limits_{\nu \to +\infty} \E^{P^\nu}\big[h^\nu(\xi)\big] =+\infty~~~ \mbox{ or } ~~~\E^P\big[ \big(\ilim\limits_{(\nu,\zeta)\to(+\infty,\xi)} h^\nu(\zeta)\big)_+ \big] <+\infty.
\]
\end{lemma}
\begin{proof}
In view of \eqref{eq:uief},
\[
\ilim\limits_{\nu \to +\infty} \E^{P^\nu}\big[  h^\nu(\xi)\h\big\{\xi\,:\, h^\nu(\xi)\le-K_0\big\}\big]>-\infty
\]
for some real $K_0>0$. Therefore,
\[
\begin{aligned}
&\slim\limits_{\nu \to +\infty} \E^{P^\nu}\big[h_-^\nu(\xi)\big]
\\ &\le \slim\limits_{\nu \to +\infty} \E^{P^\nu}\big[ h_-^\nu(\xi)\h\big\{\xi\,:\, h_-^\nu(\xi)\ge K_0\big\}\big] + \slim\limits_{\nu \to +\infty} \E^{P^\nu}\big[ h_-^\nu(\xi)\h\big\{\xi\,:\, h_-^\nu(\xi)< K_0\big\}\big]\\
&\le -\ilim\limits_{\nu \to +\infty} \E^{P^\nu}\big[ h^\nu(\xi)\h\big\{\xi\,:\, h^\nu(\xi)\le -K_0\big\}\big] + K_0<+\infty.
\end{aligned}
\]
Thus, $\slim_{\nu \to +\infty} \E^{P^\nu}\big[h_-^\nu(\xi)\big] <+\infty,$ and this inequality implies
\begin{equation}\label{eq:new3}
\ilim\limits_{\nu \to +\infty} \E^{P^\nu}\big[h_+^\nu(\xi)\big]\le \slim\limits_{\nu \to +\infty} \E^{P^\nu}\big[h_-^\nu(\xi)\big] + \ilim\limits_{\nu \to +\infty} \E^{P^\nu}\big[h^\nu(\xi)\big].
\end{equation}
Therefore, due to \eqref{eq:new3}, if $\ilim_{\nu \to +\infty} \E^{P^\nu}\big[h^\nu(\xi)\big] <+\infty$, then
$\ilim_{\nu \to +\infty} \E^{P^\nu}\big[h_+^\nu(\xi)\big] < +\infty.$ Thus, according to Fatou's lemma for weakly convergent probabilities \cite[Theorem~2.4]{FKL18},
\[
\E^P\big[ \ilim_{(\nu,\zeta)\to(+\infty,\xi)} h_+^\nu(\zeta) \big]
\le \ilim\limits_{\nu \to +\infty} \E^{P^\nu}\big[h_+^\nu(\xi)\big]<
+\infty.
\]
Therefore,
\[
\E^P\big[ \big(\ilim\limits_{(\nu,\zeta)\to(+\infty,\xi)} h^\nu(\zeta)\big)_+ \big] = \E^P\big[ \ilim\limits_{(\nu,\zeta)\to(+\infty,\xi)} h_+^\nu(\zeta) \big]<+\infty
\]
because
\[
\big(\ilim\limits_{k\to\infty}a^k\big)_+=\ilim\limits_{k\to\infty}a^k_+
\]
for each sequence $\{a^k\}_{k\in\mathbb{N}}\subset\RR$ since
\[
(\inf\limits_{k\ge n}a^k\big)_+=\inf\limits_{k\ge n}a^k_+
\]
for all $n\in\N$ and, therefore,
\[
(\sup\limits_{n\in\N}\inf\limits_{k\ge n}a^k\big)_+=\sup\limits_{n\in\N}\inf\limits_{k\ge n}a^k_+.
\]
\end{proof}

\begin{proof}[Proof of Theorem~\ref{th:LFatou}]
If $\ilim_{\nu \to +\infty} \E^{P^\nu}\big[h^\nu(\xi)\big] =+\infty$ then  \eqref{eq:new1} holds due to the convention 
$+\infty - \alpha = +\infty$ for each $\alpha \in \RR.$
Otherwise, if $\ilim_{\nu \to +\infty} \E^{P^\nu}\big[h^\nu(\xi)\big] <+\infty,$ then Lemma~\ref{lm:new} implies that
\begin{equation}\label{eq:new2}
\E^P\big[ \big(\ilim\limits_{(\nu,\zeta)\to(+\infty,\xi)} h^\nu(\zeta)\big)_+ \big]<+\infty.
\end{equation}
Fatou's lemma for weakly converging probabilities \cite[Theorem~2.4]{FKL18}, \eqref{eq:uief}, and \eqref{eq:new2} imply \eqref{eq:new1}.
\end{proof}

Let us consider the following corollary to Theorem~\ref{th:LFatou}.
\begin{corollary}\label{th:LFC}
Let $\Xi$ be a metric space, $\{P^\nu\}_{\nu\in\N}$ be a sequence of probabilities on $\Xi$ converging weakly to a probability $P$ on  $\Xi,$ and
	$\{ h^\nu\}_{\nu\in\N}$ be a sequence of measurable
	$\RR$-valued functions on $\Xi$ such that
\begin{equation}\label{eq:uiefc}
\slim_{K\to+\infty}\slim_{\nu \to +\infty} \E^{P^\nu}\big[  h^\nu(\xi)\h\big\{\xi\,:\, h^\nu(\xi)\ge K\big\}\big]=0.
\end{equation}
Then
\begin{equation}\label{eq:new1c}
\slim\limits_{\nu \to +\infty} \E^{P^\nu}\big[h^\nu(\xi)\big]\le  \E^P\big[ \slim_{(\nu,\zeta)\to(+\infty,\xi)} h^\nu(\zeta) \big].
\end{equation}
\end{corollary}
\begin{proof}
The proof repeats the proof of Theorem~\ref{th:LFatou} with minor technical modifications.
If \[
\slim\limits_{\nu \to +\infty} \E^{P^\nu}\big[h^\nu(\xi)\big] =-\infty,
\]
then  \eqref{eq:new1c} holds trivially. Otherwise, if
\[
\slim\limits_{\nu \to +\infty} \E^{P^\nu}\big[h^\nu(\xi)\big] >-\infty,
\]
then Lemma~\ref{lm:new} applied to the sequence $\{ -h^\nu\}_{\nu\in\N}$ implies that
\begin{equation}\label{eq:new2c}
\E^P\big[ \big(\slim\limits_{(\nu,\zeta)\to(+\infty,\xi)} h^\nu(\zeta)\big)_- \big]>-\infty.
\end{equation}
Fatou's lemma for weakly converging probabilities \cite[Theorem~2.4]{FKL18}, \eqref{eq:uiefc}, and \eqref{eq:new2c} imply \eqref{eq:new1c}.
\end{proof}
\begin{proof}[Proof of Theorem~\ref{th:weak-axill}]
Fix an arbitrary $x\in\X.$ Consider a sequence $\{x^\nu\}_{\nu\in\mathbb{N}}$ converging to $x$ such that
\begin{equation}\label{eq:weak1}
\ilim\limits_{(\nu,y)\to(+\infty,x)} E^{P^\nu}\big[f^\nu\big](y)=\ilim\limits_{(\nu,x^\nu)\to(+\infty,x)} E^{P^\nu}\big[f^\nu\big](x^\nu) = \ilim\limits_{\nu \to +\infty} E^{P^\nu}\big[f^\nu\big](x^\nu).
\end{equation}
Let us set $h^\nu(\,\cdot\,):=f^\nu(\,\cdot\,,x^\nu),$ $\nu\in\mathbb{N}.$ We observe that \eqref{eq:ui} implies \eqref{eq:uief}.
Therefore, since
\[
\ilim\limits_{(\nu,\zeta)\to(+\infty,\xi)} h^\nu(\zeta)\ge \ilim\limits_{(\nu,y,\zeta)\to(+\infty,x,\xi)} f^\nu(\zeta,y),
\]
\eqref{eq:weak1}, Theorem~\ref{th:LFatou},  and \eqref{eq:conv1} imply \eqref{eq:conv2}.
\end{proof}

Assumption (iii) in Theorem~\ref{th:weak-axill} relates to lsc as follows. For each $x\in\X$, suppose that the functions $\{f^\nu(\,\cdot\,,x), \nu\in\mathbb{N}\}$ are {\em equi-lsc}, i.e., for $P$-a.e. $\xi\in\Xi,$ and for each $\varepsilon > 0$ there exists $\delta>0$ such that
\begin{equation*}
f^\nu(\zeta,x) > f^\nu(\xi,x) - \varepsilon \qquad \text{for all}\  \zeta\in \ball_\Xi(\xi;\delta) \text{ and for sufficiently large } \nu\in\mathbb{N},
\end{equation*}
where $\ball_\Xi(\xi;\delta):=\{\zeta\in\Xi\,:\,d_\Xi(\zeta,\xi)<\delta\}.$ Then \eqref{eq:conv1} can be replaced with {\em lower semiconvergence} of $\{f^\nu(\,\cdot\,,x), \nu\in\mathbb{N}\}$ to $f(\,\cdot\,,x)$ in probability $P,$ that is, for each $\varepsilon > 0$, one has
	\begin{align}\label{eq:measure}
		P \big(\{ \xi \in\Xi : f^\nu(\xi,x) \leq f(\xi,x) - \varepsilon \}\big) \to 0 ~~\text{ as } \nu\to +\infty.
		\end{align}
This follows from \cite[Theorem~4.1]{FKL2}. A sufficient condition for \eqref{eq:measure} is
\begin{equation}\label{eq:poinwise}
\ilim_{\nu \to +\infty} f^\nu(\xi,x) \ge f(\xi,x) \quad \mbox{for }P\mbox{-a.e. }\xi\in\Xi.
\end{equation}

Moreover, \eqref{eq:measure} implies that
\begin{equation}\label{eq:poinwise_subsequence}
\ilim_{k\to+\infty} f^{\nu_\kappa}(\xi,x) \ge f(\xi,x) \quad \mbox{for }P\mbox{-a.e. }\xi\in\Xi,
\end{equation}
for a subsequence $\{ f^{\nu_\kappa}\}_{k\in\mathbb{N}}\subset\{f^\nu\}_{\nu\in\mathbb{N}}.$\\

According to \cite[Theorem~2.6]{FKL18}, a sufficient condition for \eqref{eq:ui} is that
for each $\bar{x}\in\X$ there exists $\rho >0$ and a sequence of measurable
$\overline{\R}$-valued functions $\{g^\nu\}_{\nu\in\mathbb{N}}$ on $\Xi$ such that
\[
\min\Big\{0,\inf\limits_{x\in\ball_\X(\bar{x};\rho)} f^\nu(\xi,x)\Big\}\ge g^\nu (\xi)
\]
for sufficiently large $\nu\in\mathbb{N}$ and $\xi\in\Xi,$	where $\ball_\X(\bar{x};\rho):=\{x\in\X:d_\X(\bar{x},x)\le\rho\}$, and
\begin{align}\label{eq:sw2aaa:var}
		-\infty< \E^P\big[\slim_{(\nu,\zeta)\to (+\infty,\xi)} g^\nu(\zeta)\big]
		\le\ilim_{\nu \to +\infty} \E^{P^\nu}\big[ g^\nu(\xi)\big].
	\end{align}
As demonstrated in \cite[Example~3.3]{FKL18}, condition \eqref{eq:ui} cannot be replaced in Theorem~\ref{th:weak-axill} with the following inequalities, which are weaker than \eqref{eq:sw2aaa:var}:
\begin{align}\label{eq:sw2aaa:varinf}
		-\infty<\E^P\big[ \ilim_{(\nu,\zeta)\to (+\infty,\xi)} g^\nu(\zeta)\big]
		\le\ilim_{\nu \to +\infty}\E^{P^\nu}\big[  g^\nu(\xi)\big].
	\end{align}

If $\{P^\nu\}_{\nu\in\mathbb{N}}$ converges to $P$ setwise, then the following inequalities
\begin{align}\label{eq:sw2aaa:var1}
		-\infty<\E^P\big[ \slim_{\nu \to +\infty} g^\nu(\xi)]
		\le\ilim_{\nu \to +\infty}\E^{P^\nu}\big[  g^\nu(\xi)\big],
\end{align}
which are similar to \eqref{eq:sw2aaa:var}, imply the validity of Fatou's inequality \cite[Theorem 4.1]{FKZ2014}.

If $P^\nu=P,$ then $P^\nu\to P$ setwise trivially.  If   $g^\nu=g,$ then  \eqref{eq:sw2aaa:var1} becomes
$\E^P\big[ g(\xi)\big]>-\infty,$ which is the existence of an integrable minorant.
This is a stronger condition than the uniform integrability condition \eqref{eq:ui} even in the classic case, when $\X$ is a singleton. In this special, case when $P^\nu = P$ and $f^\nu = f$, \cite[Proposition 8.55]{primer} states that if $f$ is random lsc and locally inf-integrable, then $f$ is lsc. In terms of the present paper, this is the special case when $g=f.$ If $P^\nu\to P,$ in total variation, then additional results on the convergence of expectations can be obtained from the uniform Fatou's lemma~\cite[Theorem 2.1]{UFL}.

\begin{theorem}\label{th:epi}{\rm(epi-convergence under weakly convergent probabilities)} Let $\X$ and $\Xi$ be metric spaces.
For probability measures $\{P, P^\nu, \nu\in \N\}$ on $(\Xi, \B(\Xi))$ and functions $\{f,f^\nu: \Xi\times \X\to \RR\}$, suppose that $P^\nu$ converges weakly to $P$, assumptions (i)-(iii) of Theorem~\ref{th:weak-axill} hold, and for each $x\in \X$ either $E^P[f](x)=+\infty$ or there exists a sequence $\{x^{\nu}\}_{\nu\in\mathbb{N}}\subset\X$ converging to $x$ with the following properties:
\begin{itemize}
\item[{\rm(i)}] one has
\begin{equation}\label{eq:uiapper}
\slim_{K\to+\infty}\slim_{\nu \to +\infty} \E^{P^\nu}\big[  f^{\nu}(\xi,x^\nu)\h\big\{\xi\,:\, f^{\nu}(\xi,x^\nu)\ge K\big\}\big]=0;
\end{equation}

\item[{\rm(ii)}] for $P\mbox{-a.e. }\xi\in\Xi$, 
\begin{equation}\label{eq:conv1conv}
\slim_{(\nu,\zeta)\to(+\infty,\xi)} f^{\nu}(\zeta,x^{\nu}) \le f(\xi,x).
\end{equation}
\end{itemize}
Then
\begin{equation}\label{eq:th42:1}
\elim\limits_{\nu \to +\infty} E^{P^\nu}\big[{f^\nu}\big] = E^P\big[f\big].
\end{equation}
\end{theorem}
\begin{proof} Choose an arbitrary $x\in\X.$ If $E^P\big[f\big](x)=+\infty,$  then Theorem~\ref{th:weak-axill} implies $\elim\limits_{\nu \to +\infty} E^{P^\nu}\big[{f^\nu}\big](x) = E^P\big[f\big](x).$

Let $E^P\big[f\big](x)<+\infty.$
Fix  a sequence $\{x^{\nu}\}_{\nu\in\mathbb{N}}\subset\X$ converging to $x$ and satisfying \eqref{eq:uiapper} and \eqref{eq:conv1conv}. Inequality \eqref{eq:conv2} follows from Theorem~\ref{th:weak-axill}. Therefore, to prove epi-convergence it is sufficient to establish the inequality
\begin{equation}\label{eq:newnew}
\slim\limits_{\nu \to +\infty} E^{P^\nu}\big[f^\nu\big](x^\nu)\le E^P\big[f\big](x),
\end{equation}
which follows from Corollary~\ref{th:LFC} with $h^\nu(\,\cdot\,):=f^\nu(\,\cdot\,,x^\nu),$ $\nu\in\mathbb{N}.$ We observe that \eqref{eq:uiapper} implies \eqref{eq:uiefc}, and \eqref{eq:new1c} and \eqref{eq:conv1conv} imply \eqref{eq:newnew}.
\end{proof}

Theorem~\ref{th:epi} strengthens the results of  \cite{DupacovaWets,Diem} by considering approximating functions $f^\nu$ rather than a fixed integrand $f$. Specifically, \cite[Theorem~1]{Diem} allows $\X$ to be a Hausdorff topological space and confirms epi-convergence of $E^{P^\nu}\big[f\big]$ to $E^{P}\big[f\big]$ under the assumptions that $f$ is lsc and $f(\cdot,x)$ is usc for each $x$. Moreover, $P^\nu$ converges weakly to $P$ and one has the following tightness condition:  for each $x\in \X$, for each neighborhood $V$ of $x,$ and for each $\epsilon>0,$ there exists a compact $K_\epsilon\subset \X$  such that
\begin{equation}\label{eq:tight}
 \int_{\Xi\setminus K_\epsilon} \big|f(\xi,x)\big| P^\nu (d\xi) + \int_{\Xi\setminus K_\epsilon} \big|f_V(\xi)\big| P^\nu (d\xi) < \epsilon  ~~\forall \nu,
\end{equation}
where
\[
f_V(\xi) = \ilim\limits_{\zeta\to\xi}[\inf\limits_{y\in V}f(\zeta,y)],~~~\xi\in\Xi.
\]

These assumptions of \cite[Theorem~1]{Diem}, when $\X$ is a metrizable space, imply the conditions of Theorem~\ref{th:epi}. Indeed, suppose that $f$ is lsc and $f(\cdot,x)$ is usc for each $x,$ and suppose that for a fixed $x\in \X,$ for each neighborhood $V$ of $x,$ and for each $\epsilon>0,$ there exists compact $K_\epsilon$  such that \eqref{eq:tight} hold. Then, we can verify \eqref{eq:ui} as follows. According to the proof of Theorem~1 from \cite{Diem}, the function $f_V$ is continuous. Its absolute value is bounded by some $\gamma>0$ on the compact set $K_\epsilon.$ Therefore,
\begin{equation}\label{eq:Diem1}
\begin{aligned}
\ilim_{y\to x} &\E^{P^\nu}\big[ f(\xi,y)\h\big\{\xi\,:\, f(\xi,y)\le-\gamma\big\}\big]\\
&\ge \inf_{y\in V}\E^{P^\nu}\big[ f(\xi,y)\h\big\{\xi\,:\, f(\xi,y)\le-\gamma\big\}\big]\\
&\ge \E^{P^\nu}\big[ f_V(\xi)\h\big\{\xi\,:\, f_V(\xi)\le-\gamma\big\}\big])\ge-\varepsilon,
\end{aligned}
\end{equation}
where the first inequality follows from the definition of lower limits, the second inequality holds because $f_V(\xi)\le f(\xi,y)<0$ if $f(\xi,y)\le-\gamma,$ and the last inequality follows from \eqref{eq:tight} because if $f_V(\xi)<-\gamma,$ then $\xi \in \Xi\setminus K_\epsilon.$ Thus, \eqref{eq:ui} holds. Similarly to these arguments, we can verify \eqref{eq:uiapper} with $x^\nu:=x,$ $\nu\in\mathbb{N}.$ Inequality \eqref{eq:conv1} for $f^\nu = f,$ $\nu\in\mathbb{N},$ follows from lsc of $f.$ Inequality \eqref{eq:conv1conv} for $f^\nu = f,$ $\nu\in\mathbb{N},$ and $x^\nu:=x,$ $\nu\in\mathbb{N},$ follows from usc of $f(\,\cdot\,,x).$

If the family of functions $\{f^{\nu}(\,\cdot\,,x^{\nu}), \nu\in\mathbb{N}\}$ is equi-usc, that is,  if the family of functions $\{-f^{\nu}(\,\cdot\,,x^{\nu})$, $\nu\in\mathbb{N}\}$ is equi-lsc, then \eqref{eq:conv1conv} can be replaced with {\em upper semiconvergence} of $\{f^{\nu}(\,\cdot\,,x^{\nu}), \nu\in\mathbb{N}\}$ to $f(\,\cdot\,,x)$ in measure $P,$ that is, for each  $\varepsilon > 0$
	\begin{align}\label{eq:measure_upper}
		P \big(\{ \xi \in\Xi : f^{\nu}(\xi,x^{\nu}) \geq f(\xi,x) + \varepsilon \}\big) \to 0 \text{ as } \nu\to +\infty.
		\end{align}
A sufficient condition for \eqref{eq:measure_upper} is
\begin{equation}\label{eq:poinwise_upper}
\slim_{\nu \to +\infty} f^{\nu}(\xi,x^{\nu}) \le f(\xi,x) \quad \mbox{for }P\mbox{-a.e. }\xi\in\Xi.
\end{equation}

\section{Applications}

As illustrations of the possibilities emerging from the above results, we present two applications of Theorem \ref{th:weak-axill} (parametric Fatou's lemma) and two examples of Theorem \ref{th:epi} (epi-convergence). In the following, for any set $C$, we let $\iota_C(x) = 0$ if $x\in C$ and $\iota_C(x) = +\infty$ otherwise.

\subsection{Sieves in Nonparametric Statistics}

Consistency analysis of parametric and nonparametric estimation problems in statistics is supported by the prior results. Here, we consider $M$-estimators with sieves \cite{GemanHwang.82, RoysetWets.20} that are distinguished by being selected from a collection of sets indexed by sample size and other varying quantities. Specifically, let $(\X,d_\X)$ be a metric space of extended real-valued functions defined on a closed subset $S$ of $\R^d$, let $(\Xi,\B(\Xi))$ be a measurable space, and let $(\PP,d_\PP)$ be a metric space of probability measures defined on $(\Xi, \B(\Xi))$. For given collections $\{f, f^\nu:\Xi\times \X \to \RR, \nu\in \N\}$, $\{X, X^\nu \subset \X, \nu\in\N\}$, and $\{P, P^\nu\in \PP, \nu\in \N\}$, a minimizer of the optimization problem
\begin{equation}\label{eqn:estimatorproblem}
\nnmin_{x\in X^\nu} \,\E^{P^\nu} \big[f^\nu(\xi,x)\big]
\end{equation}
is an estimator of ``true'' functions in $\X$ that, in turn, solve
\begin{equation}\label{eqn:estimatorproblemActual}
\nnmin_{x\in X} \,\E^P\big[f(\xi,x)\big].
\end{equation}
Often, $P^\nu$ is an empirical distribution produced by $\nu$ samples as an approximation of the actual distribution $P$, but many other possibilities exist; see for example \cite{Hess,Seri3}. The discrepancy between $X^\nu$ and $X$ stems from the need to consider restrictions and approximations of the function class $\X$, especially when $\nu$ is small and/or $d$ is large. The ``loss function'' $f$ is approximated by $f^\nu$ to accommodate computations. For example, a nonsmooth hinge-loss might be replaced by an approximating soft-max loss. The consistency of estimators obtained from \eqref{eqn:estimatorproblem} can be established through epi-convergence; see \cite{DupacovaWets,Hess,Seri3}. However, the focus has been on the special case $f^\nu = f$ and $X^\nu = X$ for all $\nu$; an exception is \cite{RoysetWets.20} where  an adjustment of \eqref{eqn:estimatorproblem} is proposed for $X^\nu \neq X$ under the assumption that $X = \X$ and $f^\nu = f$. The source of the difficulty caused by $X^\nu \neq X$ is gleaned from the assumptions in Theorem \ref{th:epi}, which essentially insist on $f^\nu(\xi^\nu,x^\nu) \to f(\xi,x)$ when $\xi^\nu\to \xi$ and $x^\nu\to x$. We discuss a different approach leveraging Theorem \ref{th:weak-axill}.

We equip $\X \times \PP$ with the product topology. The actual problem \eqref{eqn:estimatorproblemActual} is equivalently expressed as
\begin{equation}\label{eqn:estimatorproblemActual2}
\nnmin_{(x,Q)\in \X \times \PP} \,\phi(x,Q) := \E^Q\big[f(\xi,x) + \iota_X(x)\big] + \iota_{\{P\}}(Q).
\end{equation}
This reformulation motivates an alternative to \eqref{eqn:estimatorproblem} involving a nonnegative parameter $\theta^\nu\in [0, +\infty)$:
\begin{equation}\label{eqn:estimatorproblem2}
\nnmin_{(x,Q)\in \X \times \PP} \,\phi^\nu(x,Q) := \E^Q\big[f^\nu(\xi,x) + \iota_{X^\nu}(x)\big] + \theta^\nu d_\PP(Q,P^\nu);
\end{equation}
see \cite{ChenRoyset.22} for related reformulations in finite dimensions.
We show via  Theorem \ref{th:weak-axill} that estimators obtained as the $x$-component of minimizers of \eqref{eqn:estimatorproblem2} are consistent, in the sense that $\phi^\nu$ epi-converges to $\phi$ (as functions on $\X\times \PP$), under broad conditions. In particular, the analysis is unencumbered by $X^\nu \neq X$. These sets may even depend on $\xi$, but the details are omitted below. Regardless, we observe that it is advantageous to keep $\iota_{X^\nu}(x)$ inside the expectation defining $\phi^\nu$ because it helps with assumption (ii) in Proposition \ref{prop:consistency}.

We recall that the {\em outer limit} of a sequence of sets $\{A^\nu, \nu\in \N\}$ in a topological space, denoted by $\outlim A^\nu$, is the collection of points to which a subsequence of $\{a^\nu \in A^\nu, \nu\in \N\}$ converges. The {\em inner limit}, denoted by $\innlim A^\nu$, is
the collection of points to which a sequence $\{a^\nu \in A^\nu, \nu \in \N\}$ converges. If both the inner limit and the outer limit are equal to $A$, we say that $\{A^\nu, \nu \in \N\}$ {\em  set-converges} to $A$.

\begin{proposition}\label{prop:consistency}{\rm(consistency of sieved $M$-estimator)}
In the notation of this subsection, suppose that $\{X^\nu, \nu\in \N\}$ are closed and set-converge to $X$, $d_\PP$ metrizes weak convergence, $\theta^\nu \to +\infty$, $\theta^\nu d_\PP(P^\nu,P)\to 0$,  and the following hold for each $x\in \X$:
\begin{itemize}
\item[{\rm(i)}] $f(\,\cdot\,,x)$ and $f^\nu(\,\cdot\,,x)$ are measurable  for all $\nu\in \N$ and $f(\xi,x) > - \infty$ for all $\xi\in \Xi$;

\item[{\rm(ii)}] for $Q^\nu \to P$
\begin{equation}
\ilim_{K\to+\infty}\ilim_{(\nu,y)\to(+\infty,x)} \E^{Q^\nu}\big[ \big(f^\nu(\xi,y)+\iota_{X^\nu}(y)\big)\h\big\{\xi\,:\, f^\nu(\xi,y) + \iota_{X^\nu}(y)\le-K\big\}\big]=0;
\end{equation}

\item[{\rm(iii)}] for $P\mbox{-a.e. }\xi\in\Xi$
\begin{equation}
\ilim_{(\nu,y,\zeta)\to(+\infty,x,\xi)} f^\nu(\zeta,y) \ge f(\xi,x);
\end{equation}

\item[{\rm(iv)}] for $\{x^\nu\in X^\nu, \nu\in \N\}$ converging to $x$ and  $\{Q^\nu\in \PP, \nu\in \N\}$ converging to $Q \in \PP$,
\[
\ilim_{\nu \to +\infty} \E^{Q^\nu}\big[f^\nu(\xi,x^\nu)\big] > - \infty;
\]

\item[{\rm(v)}] for $\{x^\nu\in X^\nu, \nu\in \N\}$ converging to $x$, there exists a $P$-integrable function $g:\Xi\to [0,+\infty)$ such that, for $P\mbox{-a.e. }\xi\in\Xi$, one has $f^\nu(\xi,x^\nu) \leq g(\xi)$ for all $\nu\in \N$ and $f^\nu(\xi,x^\nu) \to f(\xi,x)$ as $\nu\to +\infty$.

\end{itemize}
Then, $\phi^\nu$ epi-converges to $\phi$.
\end{proposition}
\begin{proof}
First, we consider the liminf-condition required for epi-convergence. Let $\{(x^\nu,Q^\nu) \in \X \times \PP, \nu\in \N\}$ converge to $(\bar x,Q) \in \X \times \PP$. We consider three cases.

(a) Suppose that $\bar x\in X$ and $Q = P$. Without loss of generality, we assume that $x^\nu \in X^\nu$ because otherwise $\phi^\nu(x^\nu,Q^\nu) = +\infty$. Let $\hat f, \hat f^\nu:\Xi\times \X\to \RR$ be defined by $\hat f(\xi,x) := f(\xi,x) + \iota_X(x)$ and $\hat f^\nu(\xi,x) := f^\nu(\xi,x) + \iota_{X^\nu}(x)$. Then, for each $x\in \X$, $\hat f(\cdot, x)$ and $\hat f^\nu(\cdot, x)$ are measurable; recall that $X^\nu$ is closed and $X$ as well by virtue of being a set-limit. Let $x\in \X$. For $P$-a.e. $\xi\in \Xi$, one has
\[
\ilim_{(\nu,y,\zeta) \to (+\infty,x,\xi)} \hat f^\nu(\zeta,y) \geq  \ilim_{(\nu,y,\zeta) \to (+\infty,x,\xi)} f^\nu(\zeta,y) + \ilim_{(\nu,y) \to (+\infty,x)} \iota_{X^\nu}(y) \geq f(\xi,x) + \iota_X(x) = \hat f(\xi,x).
\]
The last inequality follows by (iii) and the fact that $X^\nu$ set-converges to $X$. Thus, we can invoke Theorem \ref{th:weak-axill} for the integrands $\{\hat f,\hat f^\nu, \nu\in \N\}$ and probabilities $\{P,Q^\nu, \nu\in\N\}$ and conclude that
\[
\ilim_{\nu\to +\infty} \phi^\nu(x^\nu,Q^\nu) \geq \ilim_{\nu\to +\infty} \E^{Q^\nu}\big[f^\nu(\xi,x^\nu) + \iota_{X^\nu}(x^\nu)\big] = \ilim_{\nu\to +\infty} E^{Q^\nu}[\hat f^\nu](x^\nu) \geq E^P[\hat f](\bar x) = \phi(\bar x,P).
\]

(b) If $\bar x\not\in X$, then $\phi(\bar x,Q) = +\infty$ and we seek to establish that $\phi^\nu(x^\nu,Q^\nu) \to +\infty$. Since $X^\nu$ set-converges to $X$, $x^\nu \not\in X^\nu$ for sufficiently large $\nu$. This implies that $\phi^\nu(x^\nu,Q^\nu) = +\infty$ for sufficiently large $\nu$ because $\iota_{X^\nu}(x^\nu) = +\infty$.

(c) If $\bar x\in X$ and $Q\neq P$, then again $\phi(\bar x,Q) = +\infty$. Since $\{d_\PP(Q^\nu,P^\nu), \nu\in \N\}$ is bounded away from zero as $\nu\to +\infty$ because $d_\PP(Q^\nu,P^\nu) \geq d_\PP(Q,P) - d_\PP(Q,Q^\nu) -  d_\PP(P,P^\nu)$, one has
\[
\ilim_{\nu\to +\infty} \phi^\nu(x^\nu,Q^\nu) \geq \ilim_{\nu\to +\infty} \E^{Q^\nu}\big[f^\nu(\xi,x^\nu) + \iota_{X^\nu}(x^\nu)\big] + \ilim_{\nu\to +\infty} \theta^\nu d_\PP(Q^\nu,P^\nu) = +\infty.
\]
Here, the inequality follows from (iv).

Second, we consider the limit-condition in the definition of epi-convergence. Let $(\bar x,Q) \in \X \times \PP$. In light of the above liminf-condition, it suffices to show that $\slim_{\nu\to +\infty} \phi^\nu(x^\nu,Q^\nu) \leq \phi(x,Q)$ for some $x^\nu\to x$ and $Q^\nu\to Q$. Without loss of generality, we assume that $\bar x\in X$ and $Q = P$ because otherwise the limsup-condition holds trivially. Since $X^\nu$ set-converges to $X$, there exists $\{x^\nu \in X^\nu, \nu\in \N\}$ converging to $\bar x$. We construct $\{Q^\nu = P, \nu \in \N\}$. Then,
\[
\slim_{\nu\to +\infty} \phi^\nu(x^\nu,Q^\nu) \leq \ilim_{\nu\to +\infty} \E^{P}\big[f^\nu(\xi,x^\nu)\big] + \slim_{\nu\to +\infty} \theta^\nu d_\PP(P,P^\nu) \leq \E^P\big[f(\xi,\bar x)\big] = \phi(\bar x,P)
\]
because $\theta^\nu d_\PP(P,P^\nu) \to 0$. The last inequality follows from (v), which in conjunction with (ii), permits us to invoke the dominated convergence theorem.
\end{proof}

As a concrete example, we let $\X$ be the space of extended real-valued functions on $S\subset\R^d$, excluding the function identical to $-\infty$. This is a Polish space under the hypo-distance; see \cite{RoysetWets.20}. Let $\Xi = S$ and $f^\nu(\xi,x) = f(\xi,x) = -\log x(\xi)$, which corresponds to a maximum likelihood estimator of a probability density on $S$. In this case, the requirement in (iii) holds because $x^\nu$ converging to $x$ in the hypo-distance implies that $\slim_{\nu\to +\infty} x^\nu(\xi^\nu) \leq x(\xi)$ for any $\xi^\nu \in S\to \xi$. The requirement in (v) about $f(\xi,x^\nu) \to f(\xi,x)$ when $x^\nu$ converges to $x$ in the hypo-distance now translates into having $x^\nu(\xi) \to x(\xi)$, which holds if $\X$ is equi-usc; see \cite{RoysetWets.20} for details and examples.

\subsection{Mollifiers for Discontinuous Functions}

For a normed linear space $\X$, it can be challenging to minimize a function $g:\X\to \RR$ if it is not continuous. A standard approach is to consider mollifiers to ``smooth'' the function sufficiently, i.e., to replace $g$ by the approximation $x\mapsto \E^{P^\nu}[g(x+\xi)]$, where the probabilities $P^\nu$ converge weakly to $P$, with $P(\{0\}) = 1$. Typically, these probabilities are assumed to be absolutely continuous relative to a Lebesgue measure; see \cite{ENW}. One would establish epi-convergence of the approximations to $g$ as $\nu\to +\infty$ to justify the approach. We show that Theorem \ref{th:weak-axill} can be brought in to justify approaches of this kind under general conditions including approximating $g^\nu:\X\to \RR$ and discontinuous $g$. In contrast, \cite{ENW} assumes that for each $x\in \X$ there exist points $x^\nu\to x$ such that $g(x^\nu) \to g(x)$ and then also in the special case of $\X = \R^n$.

Suppose that $\{g^\nu:\X\to \RR, \nu \in \N\}$ and $(\PP,d_\PP)$ is a metric space of probabilities defined on $(\X, \B_\X)$, where $\B_\X$ is the Borel sigma-algebra on $\X$. For a collection  $\{P^\nu\in \PP, \nu\in \N\}$ and $\theta^\nu\in [0, +\infty)$, we consider the problem
\begin{equation}\label{eqn:moll}
\nnmin_{(x,Q)\in \X \times \PP} \,\phi^\nu(x,Q) := \E^Q\big[g^\nu(x+\xi)\big] + \theta^\nu d_\PP(Q,P^\nu)
\end{equation}
as a substitute for the actual problem of minimizing $g$ over $\X$ or, equivalently, solving
\begin{equation}\label{eqn:mollactual}
\nnmin_{(x,Q)\in \X \times \PP} \,\phi(x,Q) := \E^Q\big[g(x+\xi)\big] + \iota_{\{P\}}(Q).
\end{equation}

We show via  Theorem \ref{th:weak-axill} that $\phi^\nu$ epi-converges to $\phi$ (as functions on $\X\times \PP$), which justifies the solution of \eqref{eqn:moll} in lieu of \eqref{eqn:mollactual}.

\begin{proposition}\label{prop:mollifers}{\rm(mollifiers)}
In the notation of this subsection, suppose that $d_\PP$ metrizes weak convergence, $\theta^\nu \to +\infty$, $\theta^\nu d_\PP(P^\nu,P)\to 0$, and the following hold for each $x\in \X$:
\begin{itemize}
\item[{\rm(i)}] $\{g, g^\nu, \nu\in \N\}$ are measurable;

\item[{\rm(ii)}] for $Q^\nu \to Q \in \PP$
\begin{equation}
\ilim_{K\to+\infty}\ilim_{(\nu,y)\to(+\infty,x)} \E^{Q^\nu}\big[ \big(g^\nu(y+\xi)\big)\h\big\{\xi\,:\, g^\nu(y+\xi)\le-K\big\}\big]=0;
\end{equation}

\item[{\rm(iii)}] one has
\begin{equation}
\ilim_{(\nu,y)\to(+\infty,x)} g^\nu(y) \ge g(x);
\end{equation}

\item[{\rm(iv)}] for $Q\in \PP$, $\E^Q[g(x + \xi)]>-\infty$;

\item[{\rm(v)}] there exists a $P$-integrable function $h:\Xi\to [0,+\infty)$ such that, for $P\mbox{-a.e. }\xi\in\Xi$, one has $g^\nu(x+\xi) \leq h(\xi)$ for all $\nu\in \N$ and $g^\nu(x) \to g(x)$ as $\nu\to +\infty$.

\end{itemize}
Then, $\phi^\nu$ epi-converges to $\phi$.
\end{proposition}
\begin{proof}
First, we consider the liminf-condition required for epi-convergence. Let $\{(x^\nu,Q^\nu) \in \X \times \PP, \nu\in \N\}$ converge to $(\bar x,Q) \in \X \times \PP$. Let $f, f^\nu:\Xi\times \X\to \RR$ be defined by $f(\xi,x) := g(x+\xi)$ and $f^\nu(\xi,x) := g^\nu(x+\xi)$. Then, for each $x\in \X$, $f(\cdot, x)$ and $f^\nu(\cdot, x)$ are measurable. Let $x\in \X$. For all $\xi\in \Xi$, one has
\[
\ilim_{(\nu,y,\zeta) \to (+\infty,x,\xi)} f^\nu(\zeta,y) =  \ilim_{(\nu,y,\zeta) \to (+\infty,x,\xi)} g^\nu(y+\zeta) \geq g(x+\xi) = f(\xi,x)
\]
by assumption (iii). Thus, we can invoke Theorem \ref{th:weak-axill} for the integrands $\{f,f^\nu, \nu\in \N\}$ and probabilities $\{Q,Q^\nu, \nu\in\N\}$ and conclude that
\[
\ilim_{\nu\to +\infty} E^{Q^\nu}[f^\nu](x^\nu) \geq E^Q[f](\bar x).
\]
The right-hand side is not $-\infty$ by assumption (iv). We consider two cases. (a) Suppose that $Q = P$. Then,
\[
\ilim_{\nu\to +\infty} \phi^\nu(x^\nu,Q^\nu) \geq \ilim_{\nu\to +\infty} \E^{Q^\nu}\big[g^\nu(x^\nu+\xi)\big] = \ilim_{\nu\to +\infty} E^{Q^\nu}[f^\nu](x^\nu) \geq E^P[f](\bar x) = \phi(\bar x,P).
\]
(b) If $Q\neq P$, then $\phi(\bar x,Q) = +\infty$. Since $\{d_\PP(Q^\nu,P^\nu), \nu\in \N\}$ is bounded away from zero as $\nu\to +\infty$, one has
\[
\ilim_{\nu\to +\infty} \phi^\nu(x^\nu,Q^\nu) \geq \ilim_{\nu\to +\infty} \E^{Q^\nu}\big[g^\nu(x^\nu+\xi)\big] + \ilim_{\nu\to +\infty} \theta^\nu d_\PP(Q^\nu,P^\nu) = +\infty.
\]
Here, the first inequality follows from (iv).

Second, we consider the limit-condition in the definition of epi-convergence. Let $(\bar x,Q) \in \X \times \PP$. It suffices to show that $\slim_{\nu\to +\infty} \phi^\nu(x^\nu,Q^\nu) \leq \phi(x,Q)$ for some $x^\nu\to x$ and $Q^\nu\to Q$. Without loss of generality, we assume that $Q = P$. We construct $\{x^\nu = \bar x, Q^\nu = P, \nu \in \N\}$. Then,
\[
\slim_{\nu\to +\infty} \phi^\nu(x^\nu,Q^\nu) \leq \ilim_{\nu\to +\infty} \E^{P}\big[g^\nu(\bar x + \xi)\big] + \slim_{\nu\to +\infty} \theta^\nu d_\PP(P,P^\nu) \leq \E^P\big[g(\bar x + \xi)\big] = \phi(\bar x,P).
\]
The last inequality follows from (v), which in conjunction with (ii) permits us to invoke the dominated convergence theorem.
\end{proof}

\subsection{PDE-constrained Optimization}

We next turn to the application of Theorem \ref{th:epi} in the context of PDE-constrained optimization and again assume that $\X$ and $\Xi$ are metric spaces. Suppose that we are faced with a PDE parameterized by $\xi\in \Xi$, representing parameters subject to uncertainty, and $x\in \X$ representing control input. We assume that for each $\xi\in \Xi$ and $x\in \X$, there is a unique solution of the PDE denoted by $s(\xi,x)$ in a metric space $(\U,d_\U)$, i.e., there is a well defined mapping
\[
s:\Xi\times \X\to \U
\]
furnishing solutions of the PDE for the various input parameters. Suppose that $g:\U\times \X\to \RR$ represents a quantity of interest and $P$ is a probability on $(\Xi,\B(\Xi))$ modeling the randomness in the parameter $\xi$, with $\Xi$ being a metric space and $\B(\Xi)$ its Borel sigma-algebra. Thus, a PDE-constrained optimization problem tends to involve an expectation function
\[
E^P[f]:\X\to \RR, ~~\mbox{ where the integrand } f:\Xi\times \X\to \RR \,\mbox{ has } f(\xi,x) = g\big(s(\xi,x),x\big).
\]
For computational and/or modeling reasons, the probability $P$ needs to be replaced by probabilities $P^\nu$ on $(\Xi,\B(\Xi))$, $g$ replaced by an approximation $g^\nu:\U\times\X\to \RR$, and the solution mapping $s$ by $s^\nu:\Xi\times \X\to \U$. The latter usually represents numerical solutions of the PDE. If $\X$ is infinite-dimensional, then it might need to be approximated by a set $X^\nu \subset \X$. The approximation $g^\nu$ could then simply be defined by $g^\nu(u,x) = g(u,x) + \iota_{X^\nu}(x)$. Regardless of the specifics, the setting produces an approximating expectation function
\[
E^{P^\nu}[f^\nu]:\X\to \RR, ~~\mbox{ where the integrand } f^\nu:\Xi\times \X\to \RR \,\mbox{ has } f^\nu(\xi,x) = g^\nu\big(s^\nu(\xi,x),x\big).
\]
A PDE-constrained optimization problem involving $E^P[f]$ can now be approximated using $E^{P^\nu}[f^\nu]$, with this approach being justified through epi-convergence. Theorem \ref{th:epi} provides sufficient conditions for such epi-convergence. Prior work do not consider epi-convergence in this general setting; see for example \cite{ChenRoyset.22b}, which leverages continuous convergence of expectations and thus needs relatively strong assumptions.

We show via  Theorem \ref{th:epi} that $E^{P^\nu}[f^\nu]$ epi-converges to $E^{P}[f]$ under mild assumptions.

\begin{proposition}\label{prop:mollifersKPO}{\rm(PDE-constrained optimization)}
In the notation of this subsection, suppose that $P^\nu$ converges weakly to $P$ and the following hold for each $x\in \X$:
\begin{itemize}
\item[{\rm(i)}] $\{s(\cdot, x), s^\nu(\cdot, x), g(\cdot,x), g^\nu(\cdot,x), \nu\in \N\}$ are continuous;

\item[{\rm(ii)}] one has
\begin{equation}
\ilim_{K\to+\infty}\ilim_{(\nu,y)\to(+\infty,x)} \E^{P^\nu}\big[ \big(g^\nu(s^\nu(\xi,y),y)\big)\h\big\{\xi\,:\, g^\nu(s^\nu(\xi,y),y)\le-K\big\}\big]=0;
\end{equation}

\item[{\rm(iii)}] for $P$-a.e. $\xi\in \Xi$,
\begin{align*}
\lim_{(\nu,y,\zeta)\to(+\infty,x,\xi)} s^\nu(\zeta,y) & = s(\xi,x)\\
\lim_{(\nu,v)\to(+\infty,u)} g^\nu(v,x) & = g(u,x)\\
\ilim_{(\nu,y,v)\to(+\infty,x,u)} g^\nu(v,y) & \ge g(u,x);
\end{align*}

\item[{\rm(iv)}] if $\E^P[g(s(\xi,x),x)]< \infty$, then
\begin{equation}
\slim_{K\to+\infty}\slim_{\nu\to +\infty} \E^{P^\nu}\big[ \big(g^\nu(s^\nu(\xi,x),x)\big)\h\big\{\xi\,:\, g^\nu(s^\nu(\xi,x),x)\ge K\big\}\big]=0.
\end{equation}

\end{itemize}
Then, $E^{P^\nu}[f^\nu]$ epi-converges to $E^{P}[f]$.
\end{proposition}
\begin{proof}
We first verify the assumptions of Theorem \ref{th:weak-axill} for the integrands given by $f(\xi,x) = g(s(\xi,x),x)$ and $f^\nu(\xi,x) = g^\nu(s^\nu(\xi,x),x)$. The continuity properties from assumption (i) confirm assumption (i) in Theorem \ref{th:weak-axill}. Assumptions (ii) and (iii) of Theorem \ref{th:weak-axill} hold by our assumptions (ii) and (iii), respectively.

Second, we consider the remaining assumptions of Theorem \ref{th:epi} with $x^\nu = x$ and $E^P[g(s(\xi,x),x)] <+\infty$. Assumption (i) of Theorem \ref{th:epi} holds by our assumption (iv). For $P$-a.e. $\xi\in \Xi$, our assumption (iii) implies that
\[
\slim_{(\nu,\zeta)\to(+\infty,\xi)} g^\nu\big(s^\nu(\zeta,x),x\big) \ge g\big(s(\xi,x),x\big).
\]
Thus, Assumption (ii) of Theorem \ref{th:epi} holds and the theorem applies.
\end{proof}

\subsection{Expectation Constraints}

While an expectation function might arise as the (effective) objective function in an optimization problem, and then the study of epi-convergence of its approximations becomes directly relevant, we may also have expectation functions arising as constraint functions. Still, epi-convergence remains a key tool and we again leverage Theorem \ref{th:epi}.

For continuous $f_0:\X\to \R$, integrands $\{f,f^\nu:\Xi\times \X\to \RR, \nu\in \N\}$, and probabilities $\{P,P^\nu, \nu\in \N\}$ on $(\Xi,\B(\Xi))$, suppose that we are faced with the problem
\[
\nnmin_{x\in \X} \,\phi(x) := f_0(x) + \iota_{(-\infty,0]}\big( E^P[f](x)\big).
\]
The expectation function needs to be approximated by $E^{P^\nu}[f^\nu]$. Suppose that the assumptions of Theorem \ref{th:epi} holds so that $E^{P^\nu}[f^\nu]$ epi-converges to $E^P[f]$. Then, the approximating problem
\[
\nnmin_{x\in \X} \,\phi^\nu(x) := f_0(x) + \theta^\nu\max\big\{0, E^{P^\nu}[f^\nu](x)\big\}
\]
is justified because $\phi^\nu$ epi-converges to $\phi$ under the assumption that the scalar $\theta^\nu\to +\infty$ and the following constraint qualification holds: For each $x\in \X$ satisfying $E^P[f](x) = 0$, there exist points $x^\nu \in \X \to x$ such that $E^P[f](x^\nu) < 0$.

We see this as follows. It is straightforward to confirm that $\ilim_{\nu\to +\infty} \phi^\nu(x^\nu) \geq \phi$ whenever $x^\nu \to x$. The limit condition in the definition of epi-convergence is confirmed by the following argument. Let $x\in \X$. Without loss of generality, we assume that $E^P[f](x) \leq 0$. There are two cases.

(a) Suppose that $E^P[f](x) < 0$. Since $E^{P^\nu}[f^\nu]$ epi-converges to $E^P[f]$, there exist points $x^\nu\to x$ such that
\[
\slim_{\nu\to +\infty} E^{P^\nu}[f^\nu](x^\nu) \leq E^P[f](x).
\]
Thus, for sufficiently large $\nu$, $E^{P^\nu}[f^\nu](x^\nu) \leq 0$. This means that
\begin{equation}\label{eqn:constraint}
\slim_{\nu \to +\infty} \phi^\nu(x^\nu) = \slim_{\nu \to +\infty} f_0(x^\nu) + \slim_{\nu \to +\infty} \theta^\nu\max\big\{0, E^{P^\nu}[f^\nu](x^\nu)\big\} \leq f_0(x) = \phi(x).
\end{equation}

(b) Suppose that $E^P[f](x) = 0$. By assumption, there exist points $\bar x_k \in \X \to x$ such that $E^P[f](\bar x_k) < 0$ for $k\in \N$. Fix $k$. Since $E^{P^\nu}[f^\nu]$ epi-converges to $E^P[f]$, there exist points $x_k^\nu \to \bar x_k$ as $\nu\to +\infty$ such that $\slim_{\nu\to +\infty} E^{P^\nu}[f^\nu](x_k^\nu) \leq E^P[f](\bar x_k)$. Thus, there is $\nu_k$ such that for all $\nu \geq \nu_k$, $E^{P^\nu}[f^\nu](x_k^\nu) \leq 0$. Without loss of generality, we assume that $\nu_k > \nu_l$ for $k> l$. Construct for $k \in \N$: for $\nu_{k+1}>\nu \geq \nu_k$, set $x^\nu = x_k^{\nu_k}$. This produces $\{x^\nu\}_{\nu=\nu_1}^\infty$ with the property $x^\nu \to x$. Moreover, $E^{P^\nu}[f^\nu](x^\nu) = E^{P^\nu}[f^\nu](x_k^{\nu_k}) \leq 0$ for some $k$. Thus, \eqref{eqn:constraint} holds again.

\appendix

\section{Proof of Lemma~\ref{lm:general_lsc}}\label{app}

The proof of Lemma~\ref{lm:general_lsc} uses the following properties of the Pasch-Hausdorff envelope.
\begin{lemma}\label{lm:newnew}
For each $\kappa\in [0,+\infty)$ and $h:\X\to\RR$, let $h_\kappa$ be its Pasch-Hausdorff envelope. Then the following properties are equivalent:
\begin{itemize}
\item[{\rm(a)}] $h(x) = +\infty$ for all $x\in\X$;
\item[{\rm(b)}] $h_\kappa(\bar{x})=+\infty$ for some  $\bar{x}\in\X$;
\item[{\rm(c)}] $h_\kappa(x)=+\infty$ for all $x\in\X$.
\end{itemize}
\end{lemma}
\begin{proof}
The proof follows directly from the definition of the Pasch-Hausdorff envelope in \eqref{eq:appdefiop1}.
\end{proof}

\begin{proof}[Proof of Lemma~\ref{lm:general_lsc}.] Let us  set, for $\nu\in \N$ and for $x\in\X,$
\begin{equation}\label{eq:l1}
h^\nu(x):=E^{P^\nu}[f^\nu](x).
\end{equation}

If there exists $\nu^0\in\N$ such that $h^\nu(y)=+\infty$ for all $\nu\ge\nu_0$ and for all $y\in\X,$ then, according to Lemma~\ref{lm:newnew}(a,c), $h^\nu_\kappa(y)=+\infty$ for all $y\in\X$, where $h^\nu_\kappa$ is the  Pasch-Hausdorff envelope of $h^\nu$. In this case, the conclusions of Lemma~\ref{lm:general_lsc} hold.

Otherwise, for each $\nu^0\in\N$ there exists $\nu>\nu_0$ such that $h^\nu(y^\nu)<+\infty$ for some $y^\nu\in\X.$ In this case, due to Lemma~\ref{lm:newnew}(a,b), $h^\nu_\kappa(x)<+\infty$ for all $x\in\X$ and for all $\kappa\in(0,+\infty)$. Therefore, for each $x\in \X$
\begin{equation}\label{eq:l2}
\begin{aligned}
E^{P^\nu}\big[f_\kappa ^\nu\big](x) &=
\E^{P^\nu}\Big[\inf_{y\in \X}\big\{f^\nu(\xi,y)+\kappa d_\X(y,x)\big\}\Big]\\
&\le \inf_{y\in \X} \E^{P^\nu}\big[f^\nu(\xi,y)+\kappa d_\X(y,x)\big]
=\inf_{y\in \X} \{h^\nu(y)+\kappa d_\X(y,x)\}=
h_\kappa^\nu(x),
\end{aligned}
\end{equation}
where the first equality follows from the definition of $f_\kappa ^\nu(\mathbf{\xi},x)=\inf_{y\in \X}\{f^\nu(\xi,y)+\kappa d_\X(y,x)\},$ which is $\hat{\B}^\nu(\Xi)$-measurable in $\xi$ in view of \cite[Proposition 4.2]{Hess}, the inequality holds because the expectation of the measurable infimum is not greater than the infimum of the expectation, the second equality follows from \eqref{eq:l1}, and the last equality follows from \eqref{eq:appdefiop1}.

Fix an arbitrary $x_0\in\X$. Inequalities \eqref{eq:l2} and \eqref{eq:ass1} imply that  for each $\kappa \in[\kappa_0,+\infty)$
\begin{equation}\label{eq:l3}
-\infty<E^{P}\big[f_\kappa \big](x_0)\le
\ilim\limits_{\nu \to +\infty}h_\kappa ^\nu(x_0).
\end{equation}
Therefore, there exist $\nu(x_0)\in\mathbb{N}$ and $\gamma(x_0)\in(0,+\infty)$ such that
\begin{equation}\label{eq:n2}
-\infty<-\gamma(x_0)\le h_{\kappa} ^\nu(x_0)\quad  \mbox{for each } \nu\ge \nu(x_0)\mbox{ and for each }\kappa\ge\kappa_0.
\end{equation}
Since \eqref{eq:l3} implies \eqref{eq:app5}, then, according to \eqref{eq:app6},
\begin{equation}\label{eq:l5}
\ilim\limits_{(\nu,y)\to(+\infty,x)}h^\nu(y) = \sup_{\kappa\in \N}\ilim\limits_{\nu \to +\infty}h_\kappa ^\nu(x)\quad \mbox{for each }x\in\X.
\end{equation}

Let us fix an arbitrary $\kappa> \kappa_0$ and prove that either the function $x\mapsto \ilim_{\nu \to +\infty}h_\kappa ^\nu(x)$ is identically equals to $+\infty$ or it is real-valued and Lipschitz continuous with modulus $\kappa$. Indeed, let
\begin{equation}\label{eq:n3}
\ilim\limits_{\nu \to +\infty}h_\kappa ^\nu(\tilde{x})<+\infty\mbox{ for some }\tilde{x}\in\X.
\end{equation}
For  $x\in\X$ we set
\[
I(x):=\{\nu\in\mathbb{N}\,:\,\nu\ge\nu(x_0)\mbox{ and } h_\kappa^\nu(x)<+\infty\}.
\]
Due to Lemma~\ref{lm:newnew}(b,c), $I:=I(x)=I(y)$ for all $x,y\in\X.$ Moreover, \eqref{eq:n3} implies that $I$ is an infinite set. Therefore, there exists a strictly increasing sequence $\{\nu_m\}_{m\in\mathbb{N}}\subset\N$ such  that $I=\{\nu_m\,:\, m\in\mathbb{N}\}$ and
\begin{equation}\label{eq:n5}
\ilim\limits_{\nu \to +\infty}h_\kappa ^\nu(x)=\ilim\limits_{m\to+\infty}h_\kappa ^{\nu_m}(x)
\end{equation}
for each $x\in\X$. Therefore, the definition of $I,$ \eqref{eq:n2}, and \eqref{eq:app4} imply that each function $h_\kappa^{\nu_m}$, $m\in\mathbb{N},$ is real-valued and Lipschitz continuous with modulus $\kappa.$ Therefore,
\begin{equation}\label{eq:n4}
-\infty<h_\kappa^{\nu_m}(y) - \kappa d_\X(x,y)\le h_\kappa ^{\nu_m}(x)\le h_\kappa^{\nu_m}(y)+\kappa d_\X(x,y)<+\infty
\end{equation}
for each $x,y\in\X$. Setting $y=x_0$ in \eqref{eq:n4} and passing to the lower limit as $m\to+\infty$ in the second inequality, we obtain that
\begin{equation}\label{eq:n6}
-\infty<\ilim_{m\to+\infty}h_\kappa^{\nu_m}(x_0) - \kappa d_\X(x,x_0)\le \ilim_{m\to+\infty}h_\kappa ^{\nu_m}(x)
\end{equation}
for each $x\in\X$, where the first inequality follows from \eqref{eq:l3}. Similarly, setting $y=\tilde{x}$ in \eqref{eq:n4} and passing to the lower limit as $m\to+\infty$ in the third inequality, we obtain that
\begin{equation}\label{eq:n7}
\ilim_{m\to+\infty}h_\kappa ^{\nu_m}(x)\le\ilim_{m\to+\infty}h_\kappa^{\nu_m}(\tilde{x}) + \kappa d_\X(x,\tilde{x})<+\infty
\end{equation}
for each $x\in\X$, where the second inequality follows from \eqref{eq:n3}. Finally, \eqref{eq:n5}, \eqref{eq:n6} and \eqref{eq:n7} imply that, under \eqref{eq:n3}, the function $x\mapsto \ilim_{\nu \to +\infty}h_\kappa ^\nu(x)$ is real-valued. Moreover, if we pass to the lower limit as $m\to+\infty$ in the second and the third inequalities of \eqref{eq:n4} for each $x,y\in\X$ and use equality \eqref{eq:n5} in each summand, then we obtain that the function $x\mapsto \ilim_{\nu \to +\infty}h_\kappa ^\nu(x)$ is Lipschitz continuous with modulus $\kappa$.

If the function $x\mapsto \ilim_{\nu \to +\infty}h_\kappa ^\nu(x)$ is identically equal to $+\infty$, then, according to \eqref{eq:l5}, the conclusions of Lemma~\ref{lm:general_lsc}  hold.
To finish the proof let us consider the case, when the function $x\mapsto \ilim_{\nu \to +\infty}h_\kappa ^\nu(x)$ is real-valued and Lipschitz continuous with modulus $\kappa$.

Let us prove that the function $E^{P}[f_\kappa]$ is real-valued and Lipschitz continuous with modulus $\kappa$. Indeed, since $\ilim_{\nu \to +\infty}h_\kappa ^\nu(x_0)<+\infty$, then inequalities \eqref{eq:l3} imply that $E^{P}\big[f_\kappa \big](x_0)$ is a real number. Thus,
\begin{equation}\label{eq:nn1}
f_\kappa(\xi,x_0)\mbox{ is finite for }P\mbox{-a.e. }\xi\in\Xi.
\end{equation}
Therefore, according to \eqref{eq:app4}, for $P$-a.e. $\xi\in\Xi$ the function $x\mapsto f_\kappa(\xi,x)$ is real-valued and Lipschitz continuous with modulus $\kappa$, that is,
\begin{equation}\label{eq:nn2}
f_\kappa(\xi,y)-\kappa d_\X(x,y)\le f_\kappa(\xi,x)\le f_\kappa(\xi,y)+\kappa d_\X(x,y)
\end{equation}
for each $x,y\in\X$. The function $E^{P}[f_\kappa ]$ is real-valued because
\[
-\infty<E^{P}\big[f_\kappa \big](x_0)-\kappa d_\X(x,x_0)\le E^{P}\big[f_\kappa \big](x)\le E^{P}\big[f_\kappa \big](x_0)+\kappa d_\X(x,x_0)<+\infty,
\]
where the first and the last inequalities hold because $E^{P}[f_\kappa](x_0)$ is a real number, and the second and the third inequalities follow from inequalities \eqref{eq:nn2} with $y=x_0$ if we take the expectation $\E^P$ of them. Now let us take the expectation $\E^P$ of \eqref{eq:nn2} for each $x,y\in\X$. We obtain that the function $E^{P}[f_\kappa ](x)$ is Lipschitz continuous with modulus $\kappa$.

The second inequality in \eqref{eq:l3} holds for each $x\in\X$ because both sides of this inequality are real-valued and Lipschitz with modulus $\kappa$.
Therefore,
\begin{equation}\label{eq:l4}
 -\infty< \sup_{\kappa\in \N}E^{P}\big[f_\kappa \big](x)\le \sup_{\kappa\in \N}E^{P}\big[f_\kappa \big](x) =
\ilim\limits_{(\nu,y)\to(+\infty,x)}E^{P^\nu}[f^\nu](x) \quad\mbox{for each } x\in\X,
\end{equation}
where the equality follows from \eqref{eq:l5} and \eqref{eq:l1}. Moreover, \eqref{eq:nn1}, \eqref{eq:app31} with $h :=f(\xi,\,\cdot\,),$ and lsc of $f(\xi,\,\cdot\,)$ for $P$-a.e. $\xi\in\Xi$ imply that $f_\kappa (\xi,x)\uparrow f(\xi,x)$ for each $x\in\X$ and for $P$-a.e. $\xi\in\Xi.$ Therefore, the first inequality in \eqref{eq:l4} and the monotone convergence theorem imply
\begin{equation}\label{eq:l6}
\sup_{\kappa \in \N }E^{P}\big[f_\kappa \big](x) = E^P \big[f\big](x) \quad \mbox{for each }x\in\X.
\end{equation}
Finally, \eqref{eq:l4}--\eqref{eq:l6} imply that $\ilim_{(\nu,y)\to(+\infty,x)} E^{P^\nu}[f^\nu](y)\ge E^P \big[f\big](x)>-\infty.$
 \end{proof}

\noindent {\bf Acknowledgements.} The research of the third author is supported in part by the Office of Naval Research, Science of Autonomy (N0001421WX00142) and the Air Force Office of Scientific Research, Mathematical Optimization (21RT0484).

\end{document}